\documentclass[11pt]{article}
\usepackage{amsmath,amssymb,amsthm,bbm,mathtools,mathrsfs,color,xcolor,comment}

\usepackage[left=1.4in,top=1.3in,right=1.3in,bottom=1.3in,footskip = 0.333in]{geometry}
\usepackage{mathtools}
\mathtoolsset{showonlyrefs}

\numberwithin{equation}{section}

\newtheorem{theorem}{Theorem}
\newtheorem{proposition}{Proposition}
\newtheorem{lemma}{Lemma}[section]

\newcommand{\Tr}{\mathrm{Tr}}
\newcommand{\In}{\mathcal I^n}
\newcommand{\E}{\mathbb{E}}
\newcommand{\Var}{\mathrm{Var}}

\newcommand{\dd}{\,\mathrm d}
\newcommand{\Sph}{\mathbb S^2}
\newcommand{\Splus}{\operatorname{Sym}_3^+}

\newcommand{\cM}{\mathcal M}
\newcommand{\cP}{\mathcal P}
\newcommand{\norm}[1]{\left\lVert #1\right\rVert}

\newcommand{\tr}{\text{Tr}}
\newcommand{\R}{\mathbb R}
\newcommand{\bP}{\mathcal{P}}

\newcommand{\bi}{\mathbf{i}}
\newcommand{\ba}{\mathbf{a}}
\newcommand{\snp}{\mathcal{S}^n_{\text{prod}}}
\newcommand{\U}{\mathcal U}

\begin{document}

\title{Parisi Formula for the ground state energy of quantum $p$-Spin Hamiltonians}
\date{}
\author{Sohom Bhattacharya\thanks{Department of Statistics, University of Florida, USA. Email: bhattacharya.s@ufl.edu}}
\maketitle

\begin{abstract}
    Quantum $p$-local spin glass Hamiltonians are natural quantum analogues of the classical spin glass models. We provide an asymptotic characterization of the maximal energy achievable by the product states. We prove that for every $p \ge 2$, the limit of ground state energy exists and is given by a Parisi-type variational formula. This settles a question left open by~\cite{anschuetz2025bounds}. The variational formula also recovers the known large-$p$ asymptotics as a consequence. Finally, we prove that the limiting product state energy is universal for a broad class of non-Gaussian interactions.
\end{abstract}

\section{Introduction}
We consider the problem of estimating the ground state energy of quantum $p$-local spin glass Hamiltonians~\cite{bray1980replica,erdHos2014phase}, which are quantum analogues of classical spin glass models~\cite{sherrington1975solvable,panchenko2013sherrington}. Informally, the Hamiltonian is given by
\begin{equation}
    H_{n,p}= \frac{1}{\sqrt{\binom{n}{p}}} \sum_{\substack{\sigma \in \mathcal{P}_n \\ \sigma \text{ is } p-\text{local}}} a_{\sigma} \sigma,
\end{equation}
where $\mathcal{P}_n$ denotes the set of Pauli matrices on $n$ qubits and $a_\sigma$ are i.i.d. standard Gaussian coefficients. A formal definition of the Hamiltonian is given by~\eqref{eq:hamiltonian_defn}. For classical spin glass models, seminal works by~\cite{talagrand2006parisi,panchenko2014parisi} characterize the maximal energy for every integer $p$ through a variational problem. For large $p$,~\cite{gamarnik2025shattering} has shown recently that as the number of spins $n$ diverges, the limit (scaled by $\sqrt{n}$) simplifies to $\sqrt{2 \log 2}$. However, the quantum analogue of this picture is not clearly understood.

In this article, we focus on the maximum energy achievable by the product states. Product states form a tractable variational class for local Hamiltonians; in several regimes, they provably achieve good approximations to the maximum energy~\cite{brandao2016product,bravyi2019approximation} and can even be optimized efficiently. In this article, we prove that this maximal product-state energy has a well-defined limit as $n \rightarrow \infty$ and this limit can be expressed explicitly by a Parisi-type variational formula. This answers a question raised by~\cite{anschuetz2025bounds}, who proved the limit exists for even $p$, while an explicit characterization of the fixed-$p$ limit was left open. We resolve this by identifying the limit for every fixed $p \ge 2$. As a further consequence, our Parisi-type functional independently recovers the large-$p$ asymptotics established in~\cite{anschuetz2025bounds}. Moving beyond Gaussian interactions, we prove a universality result, showing that the limiting maximal energy achievable by product states stays unchanged for a large class of interaction coefficients. 

Several other classes of random quantum Hamiltonians have been studied in physics and statistical mechanics, including the quantum Sherrington-Kirkpatrick~\cite{schindler2022variational}, quantum random energy model~\cite{baldwin2016many}, Sachdev-Ye-Kitaev (SYK) model~\cite{rosenhaus2019introduction}, for which thermodynamic limits, free energies, and ground-state behavior have been studied~\cite{crawford2007thermodynamics,baldwin2020quenched,manai2023spectral}. For some of these models, rigorous variational and Parisi-type formulas for the free energy have been established~\cite{adhikari2020free,manai2025parisi}. Ground-state energies of local quantum Hamiltonians also arise naturally as optimization problems, with product-state approximations, and random quantum Max-Cut being important examples~\cite{anshu2020beyond,brandao2016product,gharibian2019almost,parekh2022optimal,watts2024relaxations}. The random Pauli-string Hamiltonian and product-state variational problem considered here are different from these models. We do not pursue the sampling problem of finding a product state in polynomial time that attains the maximal energy, see~\cite{al2026circuit} for related circuit-complexity lower bounds for quantum $p$-spin Hamiltonians.

\noindent \textbf{Notation:}
For $n \in \mathbb N$, define $[n]=\{1,\ldots,n\}$. We denote $\mathbb S^2=\{x\in\mathbb R^3:\|x\|=1\}$.
For vectors $x,y\in\mathbb R^d$, $x\cdot y=x^\top y$, while for matrices $A,B\in\mathbb R^{d\times d}$, $A:B:=\operatorname{Tr}(AB^\top)$.
We write $I_d$ for the $d\times d$ identity matrix and $\mathrm{Sym}_d^+$ for the collection of $d\times d$ positive semidefinite
matrices. For symmetric matrices $A,B$, we write $A\succeq B$ if
$A-B\in\mathrm{Sym}_d^+$. We denote by $O(d)$ the $d\times d$ orthogonal matrices.
For a metric space $(T,d)$, $N(T,d,\varepsilon)$ denotes its
$\varepsilon$-covering number, and $\delta_x$ denotes the Dirac measure
at $x$.

\section{Main results}\label{sec:main}
Consider a system of $n$ qubits on the Hilbert space $\mathbb{C}^{2^n}$. A Hamiltonian is a Hermitian matrix $H \in \mathbb{C}^{2^n\times 2^n}$ satisfying $H^\dagger = H$. The Pauli matrices form a basis for the space of Hamiltonian matrices acting on $n$ qubits. For a single qubit, there are four Pauli matrices, known as Paulis $I, X, Y,Z$, given by:
\begin{align}\label{eq:define_paulis}
\sigma^0= I_2= \begin{bmatrix}
  1 & 0 \\
  0 & 1
\end{bmatrix}, \quad     
\sigma^1= \begin{bmatrix}
  0 & 1 \\
  1 & 0
\end{bmatrix}, \quad     
\sigma^2= \begin{bmatrix}
  0 & -i \\
  i & 0
\end{bmatrix} , \quad 
\sigma^3= \begin{bmatrix}
  1 & 0 \\
  0 & -1
\end{bmatrix}.
\end{align}
For $n$ qubits, the set of Pauli matrices $\mathcal{P}_n$ is the set of all $4^n$ possible $n-$fold tensor products of Pauli matrices
\begin{align}\label{eq:define_pauli_set}
    \mathcal{P}_n := \{\sigma^{a_1} \otimes \ldots \otimes \sigma^{a_n}: a_1,\ldots, a_n \in \{0,1,2,3\} \}.
\end{align}
Given an index vector $\bar \bi= (i_1, \ldots, i_p) \in [n]^p$ with distinct indices and $\bar \ba= (a_1,\ldots, a_p) \in \{0,1,2,3\}^p$, we define the Pauli matrix $P_{\bar\bi}^{\bar\ba}$ by
\begin{align}\label{eq:tensor_prod}
    P_{\bar\bi}^{\bar\ba}= \prod_{j=1}^p \sigma^{a_j}_{i_j},
\end{align}
where $\sigma^{a}_{i}= I^{\otimes (i-1)}_2 \otimes \sigma^a \otimes I^{\otimes (n-i)}_2$. For $n$ qubits, we use bra-ket notation to denote states where a ket $|\phi\rangle$ is a vector in the vector space $|\phi\rangle \in \mathbb{C}^{2^n}$, and $\langle \phi|$ denotes its conjugate transpose. Using this bra-ket notation, the set of all pure states is given by $\mathcal{S}^n:= \{|\phi\rangle \in \mathbb{C}^{2^n}: |\langle \phi| \phi \rangle|^2=1\}$. A state $|\phi\rangle$ is a product state on $n$ qubits if it can be written as a tensor product of states on its individual qubits; the set of product states on $n$ qubits is denoted as 
\begin{align}\label{eq:product_space}
    \snp := \Big\{ |\phi\rangle=  |\phi^{(1)}\rangle \otimes \ldots \otimes |\phi^{(n)}\rangle \in \mathcal{S}^n: |\phi^{(i)}\rangle \in \mathcal{S}^1 \quad \forall i \in [n]\Big\}.
\end{align}
Denote the set of ordered $p$-tuples by $\mathcal{I}^{n}=\{\bar\bi \in [n]^p: i_1< \ldots <i_p\}$. Fix integers $p\ge 2$ and $n\ge p$. The Hamiltonian of the quantum $p$- spin  is given by 
\begin{equation}\label{eq:hamiltonian_defn}
    H_{n,p}= N^{-1/2}_n \sum_{\bar\bi \in \mathcal{I}^n} \sum_{\bar \ba \in \{1,2,3\}^p} \alpha [\bar \bi, \bar \ba] P_{\bar\bi}^{\bar\ba},
\end{equation}
where $N_n:=\binom np$ and $\alpha [\bar \bi, \bar \ba]$ are i.i.d. standard normal coefficients. 

An important quantity of interest in statistical physics is the maximum energy of $H_{n,p}$ in the limit of large $n$. In this article, we consider the maximum energy when optimized over the set $\snp$ of all product states on $n$ qubits:
\begin{equation}\label{eq:define_groundstate}
    E^{\star}_{n,\text{prod}}(p)= \frac{1}{\sqrt{n}} \max_{|\phi\rangle \in \snp} \langle \phi| H_{n,p}|\phi\rangle.
\end{equation}
We characterize the limit of $E^{\star}_{n,\text{prod}}(p)$ for every integer $p \ge 2$ as $n \rightarrow \infty$. Of course, the existence of a limit is not immediate. It was shown in~\cite[Theorem 4]{anschuetz2025bounds} using a near super-additivity argument~\cite{de1951some} that the limit exists for even $p$. Further, for large $p$, the limit is shown to be $(1+o(1))\sqrt{2 \log p}$~\cite[Theorem 3]{anschuetz2025bounds}. Unfortunately, these results do not identify the limit for a fixed $p$. We address this gap by proving a Parisi-type formula for the limiting maximum energy for any $p \ge 2$.

To describe the limit, we need to define a few quantities. Let $\mathcal{U}$ denote the collection of all nonnegative, nondecreasing functions $\gamma$ on $[0,1)$ that are right continuous with $\int_0^1 \gamma(t) dt <\infty$. Let $(B_t)_{0 \le t \le 1}$ be a standard Brownian motion in $\mathbb R^3$. For $0 \le s<t \le 1$, denote by $D[s,t]$ the collection of all progressively measurable processes $u:[s,t]\mapsto \mathbb{R}^3$ with respect to the Brownian filtration and satisfying $ \|u(r)\| \le 1$ for almost every $r$. Throughout the paper, $\|.\|$ denotes the Euclidean norm. We define $\xi_{p}(q):=q^p$.
For $\gamma \in \mathcal{U}$, define
\begin{align}\label{eq:define_p}
    \Psi_{p,\gamma}(s,x):= \sup_{u \in D[s,1]}\Bigg\{\mathbb{E} &\Big[\Big\|x+ \int_{s}^1 \sqrt{\xi''_{p}(t)} dB_t+ \int_{s}^1 \xi''_p(t) \gamma(t)u(t) dt\Big\|\Big] \nonumber \\
    &- \frac{1}{2}\int_s^1\xi''_p(t) \gamma(t) \mathbb{E}(\|u(t)\|^2) dt \Bigg\}.
\end{align}
Finally, set 
\begin{align}\label{eq:define_variational}
    \bP_p(\gamma)= \Psi_{p,\gamma}(0,0)-\frac{1}{2} \int_{0}^1 t \xi''_p(t)\gamma(t)dt.
\end{align}
We now state our main result.
\begin{theorem}\label{thm:main}
    For every $p\ge 2$, we have $E^{\star}_{n,\text{prod}}- \mathbb{E} E^{\star}_{n,\text{prod}}= o_{\mathbb{P}}(1)$. Further,
\begin{equation}\label{eq:limit_thm}
E^\star_{\text{prod}}(p):=\lim\limits_{n\rightarrow\infty} \mathbb{E} (E^{\star}_{n,\text{prod}})= \inf_{\gamma \in \mathcal{U}}\bP_p(\gamma),
\end{equation}
where $E^{\star}_{n,\text{prod}}$ and $\bP_p(\gamma)$ are defined by~\eqref{eq:define_groundstate} and~\eqref{eq:define_variational} respectively. The infimum in the above display is attained.
\end{theorem}
Theorem~\ref{thm:main} characterizes the maximum energy attainable by product states. The proof also shows that there exists $\gamma^\star \in \mathcal{U}$ such that $E^\star_{\text{prod}}= \mathcal{P}_{p}(\gamma^\star)$.  
We do not pursue the sampling problem of finding a product state in polynomial time that attains the maximal energy. The proof proceeds by reducing the optimization problem to a maximization problem on $(\mathbb S^2)^n$ through the Bloch-sphere representation~\cite{nielsen2000quantum}. We then compare the reduced problem with a vector spin-glass maximization problem and invoke the finite-temperature Parisi formula~\cite{chen2025free}. We next show, using rotational symmetry, that the resulting matrix-valued variational problem reduces to a scalar Parisi formula. To achieve the zero-temperature limit, we adapt the argument of~\cite{auffinger2015parisi} for our setting, completing the proof.

Theorem~\ref{thm:main} leads to an alternative proof of the large $p$ asymptotics of $E^{\star}_{\text{prod}}(p)$, recovering~\cite[Theorem 3]{anschuetz2025bounds}.
\begin{proposition}\label{cor:large_p}
Suppose we are in the setup of Theorem~\ref{thm:main}. Then we have,
$$
\lim\limits_{p \rightarrow \infty}\frac{E^{\star}_{\text{prod}}(p)}
{\sqrt{2\log p}}=1.$$
\end{proposition}
Next, we demonstrate universality of the asymptotic maximal energy, showing that it remains independent of the distribution of the coefficients. To this end for $p\ge2$, with coefficients $\alpha'=(\alpha'[\bar \bi,\bar \ba])$, $\bar i \in \mathcal{I}^n$, $\bar a \in [3]^p$, define
\begin{equation}\label{eq:model}
\begin{split}
H_{n,p}(\alpha')&=N_n^{-1/2}\sum_{\bar \bi\in\In}\sum_{\bar \ba\in\{1,2,3\}^p}
 \alpha'[\bar i,\bar a]P_{\bar i}^{\bar a},\\
E^\star_{n,\mathrm{prod}}(p;x)&=\frac1{\sqrt n}
\max_{|\phi\rangle\in\snp}\langle\phi|H_{n,p}(x)|\phi\rangle.
\end{split}
\end{equation}
When $\alpha'[\bar\bi,\bar\ba]=\alpha[\bar\bi,\bar\ba]$ are i.i.d. Gaussian coefficients, we get~\eqref{eq:hamiltonian_defn}.
A precise version of our universality result is given below:
\begin{proposition}[Universality]\label{prop:universality}
Fix an integer $p \ge 2$, let the coefficients $\alpha'[\bar \bi,\bar \ba]$ in~\eqref{eq:model} be independent, with
$\E\alpha'[\bar \bi,\bar \ba]=0$ and $\E\alpha'[\bar \bi,\bar \ba]^2=1$.
Suppose that, for every $\varepsilon>0$,
\begin{equation}\label{eq:lindeberg-assumption}
L_n(\varepsilon):=\frac1{N_n}\sum_{\bar \bi\in\In}\sum_{\bar \ba\in\{1,2,3\}^p}
\E\!\left[\alpha'[\bar \bi,\bar \ba]^2\,;
 |\alpha'[\bar \bi,\bar \ba]|>\varepsilon n^{(p-1)/2}\right]
\longrightarrow0.
\end{equation}
Let $\alpha$ have independent standard Gaussian coordinates. Then
\begin{equation}\label{eq:universality}
\left|\E E^\star_{n,\mathrm{prod}}(p;\alpha')
      -\E E^\star_{n,\mathrm{prod}}(p;\alpha)\right|\longrightarrow0.
\end{equation}
Consequently, we have
\begin{equation}\label{eq:parisi-conclusion}
E^\star_{n,\mathrm{prod}}(p;\alpha')
\longrightarrow E^\star_{\mathrm{prod}}(p)
=\inf_{\gamma\in\U}\mathcal{P}_p(\gamma)
\end{equation}
in probability, where $\mathcal{P}_p$ is defined in~\eqref{eq:define_variational}.
\end{proposition}
Proposition~\ref{prop:universality} applies to sparse random hypergraphs 
if the average degree diverges with $n$. 
For the classical SK model, universality of ground-state energy is well-known~\cite{carmona2006universality}. A general invariance principle was developed by~\cite{chatterjee2005simple} and was applied to spin glasses, achieving a comparison mechanism used in many later universality arguments. For the quantum Hamiltonians,~\cite{anschuetz2025bounds} proves universality for maximum energy over all pure states, including entangled states. In contrast, Proposition~\ref{prop:universality} focuses on product states. Coupled with Theorem~\ref{thm:main}, it yields the precise limit of maximal energy through the Parisi functional. The proof of~\cite{anschuetz2025bounds} approximates $\lambda_{\max}(H_n)$ with $\beta^{-1}\log \tr e^{\beta H_n}$, which does not approximate the maximum over product states. We therefore include a separate proof for universality in Section~\ref{sec:proofs}.

\section{Proofs}\label{sec:proofs}
In this section, we prove all the results from Section~\ref{sec:main}. Throughout the proofs, $C>0$ will denote an arbitrary constant whose value will change.

\begin{proof}[Proof of Theorem~\ref{thm:main}]
 A pure one-qubit state $|\phi\rangle$ is represented by its Bloch vector~\cite[Section 1.2]{nielsen2000quantum}. Namely, there exists a unique $s= (s^1,s^2,s^3)\in \mathbb{S}^2$ such that  
\begin{align}\label{eq:bloch}
    |\phi\rangle\langle\phi|= \frac{1}{2}(I_2 +\sum_{a=1}^3 s^a \sigma^a), \qquad s^a= \langle\phi|\sigma^a|\phi\rangle.
\end{align}
Hence, a product state $|\phi\rangle =|\phi^{(1)}\rangle \otimes \ldots \otimes |\phi^{(n)}\rangle \in \snp$ is parametrized by $s=(s_1,\ldots,s_n) \in (\mathbb{S}^2)^n$. This gives the following reduction to classical Hamiltonian.

\begin{lemma}\label{lem:reduction}
    Let $|\phi\rangle = \otimes_{i=1}^n |\phi^{(i)}\rangle \in \snp$, and $s_i\in \mathbb{S}^2$ be the Bloch vector of $|\phi^{(i)}\rangle$. Then, 
    \begin{equation}\label{eq:x_proc_reduce}
        X_{n,p}(s)= \sqrt{n} \langle\phi|H_{n,p}|\phi\rangle,
    \end{equation}
    where we define
    \begin{equation}\label{eq:define_x_process}
        X_{n,p}(s)= \sqrt{\frac{n}{N_n}} \sum_{\bar\bi\in \mathcal{I}^n} \sum_{\bar \ba \in \{1,2,3\}^p} \alpha [\bar \bi, \bar \ba] \prod_{j=1}^p s^{a_j}_{i_j}.
    \end{equation}
    Further, every $s\in(\mathbb{S}^2)^n$ arises from some state in $\snp$. Consequently, 
    \begin{equation}\label{eq:to_show}
        E^{\star}_{n,\text{prod}}(p)= \frac{1}{n} \Big[\max_{s\in(\mathbb{S}^2)^n} X_{n,p}(s)\Big].
    \end{equation}
Finally, $E^{\star}_{n,\text{prod}}(p)- \mathbb{E}(E^{\star}_{n,\text{prod}}(p))=o_{\mathbb{P}}(1)$.
\end{lemma}
\begin{proof}[Proof of Lemma~\ref{lem:reduction}]
    For $\bar \bi\in \mathcal{I}^n$ and $\bar \ba \in [3]^p$, we have
    \begin{align*}
        \langle \phi|P^{\bar \ba}_{\bar \bi}|\phi \rangle= \prod_{j=1}^p  \langle \phi^{(i_j)}|\sigma^{a_j}|\phi^{(i_j)} \rangle= \prod_{j=1}^{p}s^{a_j}_{i_j}.
    \end{align*}
This proves~\eqref{eq:x_proc_reduce}. The converse follows from the definition~\eqref{eq:bloch}. Therefore, optimizing over $\snp$ is equivalent to maximizing $X_{n,p}(s)$ over ${s\in(\mathbb{S}^2)^n}$.
Finally, by standard Gaussian concentration, we have~\cite[Proposition 10]{anschuetz2025bounds},  $E^{\star}_{n,\text{prod}}(p)- \mathbb{E}(E^{\star}_{n,\text{prod}}(p))=o_{\mathbb{P}}(1).$ This completes the proof of the lemma.
\end{proof}
The above lemma turns optimization over $\snp$ into maximizing over $(\mathbb{S}^2)^n$. Next, we construct a centered Gaussian process $Y_{n,p}$. For $s=(s_1,\ldots, s_n)$, $t=(t_1,\ldots, t_n) \in (\mathbb{S}^2)^n$, define
\begin{align}\label{eq:overlap}
    R(s,t)= \frac{1}{n}\sum_{i=1}^n s_it^\top_i, \quad q(s,t)= \tr R(s,t), \quad Q(s)= R(s,s).
\end{align}
Let $Y_{n,p}$ be a centered Gaussian process on $(\mathbb{S}^2)^n$ with covariance 
\begin{equation}\label{eq:y_cov}
    \mathbb{E} Y_{n,p}(s)Y_{n,p}(t)= n \xi_p(q(s,t))= n (q(s,t))^p.
\end{equation}
The following lemma connects $X_{n,p}$ with the Gaussian process $Y_{n,p}$. For $\beta>0$ and let $\nu$ be the normalized surface measure on $\mathbb{S}^2$, define
\begin{align}\label{eq:define_ax}
    A^X_{n,p}(\beta)
:=
\frac1n
\mathbb E\log
\int_{(\mathbb{S}^2)^n}
\exp\{\beta X_{n,p}(s)\}
\,\nu^{\otimes n}(ds),
\end{align}
and define $A^Y_{n,p}(\beta)$ analogously.

\begin{lemma}\label{lem:x_covar}
Let
$|\phi\rangle
=
|\phi^{(1)}\rangle\otimes\cdots\otimes|\phi^{(n)}\rangle$
and
$|\psi\rangle
=
|\psi^{(1)}\rangle\otimes\cdots\otimes|\psi^{(n)}\rangle$
be product states in $\snp$, and let $s_i,t_i\in \mathbb{S}^2$ denote the
Bloch vectors of $|\phi^{(i)}\rangle$ and $|\psi^{(i)}\rangle$,
respectively. Then
\begin{enumerate}
    \item[(i)] The covariance of $X_{n,p}$ is 
    \begin{align}\label{eq:exact-x-covariance}
    \mathbb E X_{n,p}(s)X_{n,p}(t)
=
\frac{n}{N_n}
\sum_{\bar \bi=(i_1,\ldots,i_p)\in \mathcal{I}^n}
\prod_{r=1}^p s_{i_r}\cdot t_{i_r}.
\end{align}
Consequently,
\begin{align}\label{eq:covariance-error}
    \sup_{s,t\in(\mathbb S^2)^n}
\left|
\mathbb E X_{n,p}(s)X_{n,p}(t)
-
\mathbb E Y_{n,p}(s)Y_{n,p}(t)
\right|
\le p(p-1).
\end{align}
In particular,
$\mathbb E X_{n,p}(s)^2=\mathbb E Y_{n,p}(s)^2=n$
for every $s\in(\mathbb{S}^2)^n$. 
\item[(ii)] For $\beta>0$, 
\begin{align}\label{eq:free-energy-comparison}
    \left|
A^X_{n,p}(\beta)-A^Y_{n,p}(\beta)
\right|
\le
\frac{\beta^2p(p-1)}{2n}.
\end{align}
\end{enumerate}
\end{lemma}
\begin{proof}[Proof of Lemma~\ref{lem:x_covar}]
For product states
$|\phi\rangle,|\psi\rangle\in \snp$, we have, by~\cite[Lemma 7]{anschuetz2025bounds},
\begin{align}
    &\mathbb E
\left[ \langle\phi|H_{n,p}|\phi\rangle \langle\psi|H_{n,p}|\psi\rangle \right]
= \frac1{|\mathcal{I}^n|} \sum_{\bar \bi=(i_1,\ldots,i_p)\in \mathcal{I}^n} \prod_{r=1}^p
\left( 2\left| \langle\phi^{(i_r)}|\psi^{(i_r)}\rangle \right|^2-1 \right).
\label{eq:AGK-covariance}
\end{align}
Using~\eqref{eq:bloch}, we have
$$
\Big| \langle\phi^{(i_r)}|\psi^{(i_r)}\rangle \Big|^2= \tr (|\phi^{(i_r)}\rangle\langle\phi^{(i_r)}|
|\psi^{(i_r)}\rangle\langle\psi^{(i_r)}|)= \frac{1+s_{i_r}.t_{i_r}}{2}.
$$
This proves~\eqref{eq:exact-x-covariance}. Next, let $(I_1,\ldots,I_p)$ be uniformly
distributed on $\mathcal{I}^n$. Then
$$
\frac1{|\mathcal{I}^n|} \sum_{\bar \bi=(i_1,\ldots,i_p)\in \mathcal{I}^n} \prod_{r=1}^p s_{i_r}. t_{i_r}
= \mathbb E\prod_{r=1}^p s_{I_r}. t_{I_r}.
$$
If $J_1,\ldots,J_p$ are i.i.d. uniformly distributed
on $[n]$, then
\begin{equation*}
    \frac{1}{n}\mathbb E Y_{n,p}(s)Y_{n,p}(t)= q(s,t)^p=\mathbb E\prod_{r=1}^p s_{J_r}.t_{J_r}.
\end{equation*}
When  $J_1,\ldots,J_p$ are all distinct, couple $(I_1,\ldots, I_p)$ to their increasing rearrangement. Since $|s_{i}. t_{i}|\le1$,
\begin{align}
\left|
\mathbb E\prod_{r=1}^p s_{I_r}. t_{I_r}-q(s,t)^p\right|
&\le 2\,\mathbb P \left(J_r=J_\ell\text{ for some }r<\ell\right)
\le 2\binom p2\frac1n=\frac{p(p-1)}{n} \nonumber.
\end{align}
Multiplying by $n$ proves \eqref{eq:covariance-error}. If $s=t$, then
$s_i.t_i=1$ for every $i$, so both variances equal $n$ for every $s\in(\mathbb{S}^2)^n$. To prove~\eqref{eq:free-energy-comparison}, take independent copies of $X_{n,p}$, $Y_{n,p}$ and for $\lambda\in[0,1]$, set the interpolating process
$H_\lambda(s)
=\sqrt{\lambda}\,X_{n,p}(s) + \sqrt{1-\lambda}\,Y_{n,p}(s)$ and
$$
\varphi(\lambda)
= \frac1n
\mathbb E\log \int_{(\mathbb{S}^2)^n} \exp\{\beta H_\lambda(s)\} \,\nu^{\otimes n}(ds).
$$
Write
$\Delta C(s,t) = \mathbb E X_{n,p}(s)X_{n,p}(t) - \mathbb E Y_{n,p}(s)Y_{n,p}(t)$.
Gaussian integration by parts gives~\cite[Lemma 1.1]{panchenko2013sherrington}
$$
\varphi'(\lambda)=\frac{\beta^2}{2n}\mathbb E\left\langle\Delta C(s^1,s^1)-\Delta C(s^1,s^2)\right\rangle_\lambda,
$$
where $s^1, s^2$ are independent Gibbs samples. Since the variances of $X_{n,p}$ and $Y_{n,p}$ agree,
$\Delta C(s,s)=0$, while \eqref{eq:covariance-error} gives
$|\Delta C(s,t)|\le p(p-1)$. Therefore
$|\varphi'(\lambda)|
\le \beta^2p(p-1)/(2n)$.
Integrating over $\lambda\in[0,1]$ proves
\eqref{eq:free-energy-comparison}.
\end{proof}

Next, we show that $Y_{n,p}$ has the same limiting ground-state value as $X_{n,p}$. For $Z=X_{n,p}$ or $Z=Y_{n,p}$, write
$m^Z_{n,p}:=n^{-1}\mathbb E\max_{s\in(\mathbb{S}^2)^n}Z(s)$.

\begin{lemma}
\label{lem:annealing}
Fix an integer $p\ge2$.  There exists a deterministic function
$\varepsilon_p(\beta)$ such that $\varepsilon_p(\beta)\rightarrow0$ as $\beta\to\infty$, independent of $n$,
such that, for $Z=X_{n,p}$ or $Z=Y_{n,p}$,
\begin{align}\label{eq:beta_to_inf}
    m^Z_{n,p}-\varepsilon_p(\beta)
\le
\frac{1}{\beta}A^Z_{n,p}(\beta)
\le
m^Z_{n,p},
\end{align}
where $A^Z_{n,p}$ is defined by~\eqref{eq:define_ax}. 
Suppose that, for every fixed $\beta>0$, the limit
$F_p(\beta):=\lim_{n\to\infty} \frac{1}{\beta} A^Y_{n,p}(\beta)$ exists.  Then
$$
\lim_{n\to\infty}m^X_{n,p}
=
\lim_{n\to\infty}m^Y_{n,p}
=
\lim_{\beta\to\infty}F_p(\beta)
=
E^\star_{\rm prod}(p).
\label{eq:same-ground-state}
$$
\end{lemma}

\begin{proof}
The upper bound in \eqref{eq:beta_to_inf} is immediate from the definition of $A^Z_{n,p}$ because
$\nu^{\otimes n}$ is a probability measure.  We prove the lower bound for $X_{n,p}$, $Y_{n,p}$.

Let $d_Z(s,t)^2:=\mathbb E\{Z(s)-Z(t)\}^2$. We claim that for both $X_{n,p}$ and $Y_{n,p}$, we have
\begin{align}\label{eq:dz_euclidean}
    d_Z(s,t)^2 \le p \|s-t\|^2
\end{align}
To see this, note $d_Y(s,t)^2=2n(1-q(s,t)^p) \le 2n p(1-q(s,t))= p \|s-t\|^2$. Further, using \eqref{eq:exact-x-covariance} with $1-\prod_{r=1}^p s_{i_r}t_{i_r} \le \sum_{r=1}^p (1-s_{i_r}t_{i_r})$ yields the inequality for $X_{n,p}$.

An $\eta$-net of $\mathbb{S}^2$ has cardinality at most $(C/\eta)^2$. Hence, for $(\mathbb{S}^2)^n$, to achieve a Euclidean distance $\eta \sqrt{n}$, it is sufficient to approximate each coordinate within $\eta$. Therefore, we obtain the covering number bound
\begin{equation}
  N\bigl((\mathbb{S}^2)^n,\|\cdot\|_{\mathbb R^{3n}},\eta\bigr)
\le
\left(\frac{C n}{\eta^2}\right)^{n}.  
\end{equation}
Using~\eqref{eq:dz_euclidean}, the above display further yields that
\begin{equation}
    \log N\bigl((\mathbb{S}^2)^n,d_Z,\eta\bigr) \le 2n \log \frac{C\sqrt{pn}}{\eta}.
\end{equation}
Define
$\Omega_{n,\delta}
:=
\sup_{\|s-t\|\le\delta\sqrt n}|Z(s)-Z(t)|$. By Dudley's integral inequality applied to this process
\cite[Theorem~8.1.3]{vershynin2019high}, for $0<\delta\le 1/2$,
\begin{align}
\mathbb E
\Omega_{n,\delta}
&\le
C\int_0^{C\sqrt{pn}\delta}
\sqrt{
n\log\left(\frac{C\sqrt{pn}}{\eta}\right)
}\,d\eta \nonumber\\
&\le
C_p n\delta\sqrt{\log(C_p/\delta)}.
\label{eq:local-gaussian-bound}
\end{align}
Let $s^\star$ be a maximizer of $Z$.  Consider the set
$$B_{\delta}(s^\star):=\{t \in (\mathbb S^2)^n: \|t_i-s_i^\star\|\le\delta \textrm{ for every } i \in [n]\}.$$
Recall that if $\nu$
denotes normalized surface measure on $S^2$, then $\nu^{\otimes n} [B_{\delta}(s^\star)]=(\delta^2/4)^n$.
If $t\in B_{\delta}(s^\star)$, $\|t-s^\star\| \le \delta \sqrt{n}$. For every such $t$,
$Z(t)\ge Z(s^\star)-\Omega_{n,\delta}$.  Restricting the partition function to $B_{\delta}(s^\star)$ yields
\begin{align}
\frac{A_{n,p}^Z(\beta)}{\beta} &\ge \frac1{\beta n}
\mathbb E\log
\int_{B_{\delta}(s^\star)}
\exp\{\beta Z(s)\}
\,\nu^{\otimes n}(ds) \nonumber\\
&\ge
\frac1n\mathbb E\max_s Z(s)
-
\frac1n\mathbb E\Omega_{n,\delta}
+
\frac1\beta\log\frac{\delta^2}{4} \nonumber\\
&\ge
m_{n,p}^Z
-
C_p\delta\sqrt{\log(C_p/\delta)}
+
\frac1\beta\log\frac{\delta^2}{4}.
\label{eq:annealing-lower-bound}
\end{align}
Taking $\delta=\min\{\frac12, \beta^{-1/2}\}$, we obtain, uniformly in $n$,
\begin{align}
m_{n,p}^Z-\varepsilon_p(\beta)
\le
\frac{A_{n,p}^Z(\beta)}{\beta}
\le
m_{n,p}^Z,
\label{eq:uniform-annealing}
\end{align}
where
$\varepsilon_p(\beta)
=
C_p\delta \sqrt{\log(C_p/\delta)}
+\beta^{-1}\log(4/ \delta^2)$
and $\varepsilon_p(\beta)\to0$ as $\beta \rightarrow \infty$. This proves~\eqref{eq:beta_to_inf}.

Next, we compare the ground states of
$X_{n,p}$ and $Y_{n,p}$. For every fixed $\beta$, using Lemma~\ref{lem:x_covar},
\begin{align*}
\lim\limits_{n\rightarrow \infty}\left|
A_{n,p}^X(\beta)-A_{n,p}^Y(\beta)
\right|=0
\end{align*}
Hence, if the limit
$F_p(\beta)$ exists, the same limit holds for
$A_{n,p}^X(\beta)/\beta$. Applying~\eqref{eq:beta_to_inf}, for either
$Z=X_{n,p}$ or $Z=Y_{n,p}$ we obtain
\begin{equation}
F_p(\beta)\leq\liminf_{n\to\infty}m^Z_{n,p}\leq \limsup_{n\to\infty}m^Z_{n,p}\leq F_p(\beta)+\varepsilon_p(\beta).
\end{equation}
Since this holds for every $\beta>0$ and $\varepsilon_p(\beta)\to0$ as
$\beta\to\infty$, it follows that $\lim_{n\to\infty}m^Z_{n,p}$ exists and equals $\lim\limits_{\beta\rightarrow\infty} F_p(\beta)$, for both choices of $Z$. This completes the proof.
\end{proof}

Now, we turn our attention to finding the asymptotic limit $\lim\limits_{n\rightarrow \infty }A^Y_{n,p}(\beta)$ for any finite $\beta>0$. The asymptotic free energy for vector spin-glass models have been studied recently~\cite{bovier2009aizenman,panchenko2018free,chen2023self,chen2025free}. We need a few definitions to describe the limit. 
Let $\Pi$ be the set of left-continuous maps $\pi:[0,1]\to\Splus$ that are increasing: $\pi(t) \succeq \pi(t')$ if $t\ge t'$. Let $\mathcal R$ be the continuous Ruelle probability cascade~\cite{bolthausen1998ruelle,ruelle1987mathematical} whose overlap $\alpha\wedge\alpha'$ between two independent samples is uniform on $[0,1]$. For a  path $\pi\in\Pi$, let $w^\pi(\alpha)$ be the centered $\R^3$-valued Gaussian process defined by:
\begin{align*}
\E\bigl[w^\pi(\alpha)w^\pi(\alpha')^\top\bigr]
=\pi(\alpha\wedge\alpha').
\end{align*}
We use~\cite[Corollary 8.3]{chen2025free} for $D=3$. Suppose $\mathcal{H}_n(\sigma)$ is a Gaussian process with $\sigma\in \mathbb{R}^{d \times n}$ with covariance
\begin{equation}\label{eq:define_xi}
    \mathbb{E} \mathcal{H}_n(\sigma) \mathcal{H}_n(\sigma')= n \xi\left(\frac{1}{n}\sigma(\sigma')^\top\right).
\end{equation} Then under certain convexity and monotonicity assumptions of $\xi$, the free energy
\begin{align}
    \mathcal{F}_n= \frac{1}{n}\mathbb{E}\log\int\exp\left(\mathcal{H}_n(\sigma)- \frac{1}{2}n \xi\left(\frac{1}{n}\sigma \sigma^\top\right)\right)d\nu^{\otimes n} 
\end{align}
satisfies $\lim\limits_{n \rightarrow \infty} \mathcal{F}_n= \inf_{\pi \in \Pi} \mathscr P_{p,\beta} (\pi)$, where
\begin{align}
\mathscr P_{p,\beta}(\pi)
:=&\ \E\log\iint
\exp\left\{w^{\nabla\xi\circ\pi}(\alpha)\cdot s
-\frac12\nabla\xi(\pi(1)):ss^\top\right\}
\nu(\dd s)\mathcal R(\dd\alpha)
\nonumber\\
&\quad+\frac12\int_0^1\theta(\pi(u))\dd u,
\qquad
\theta(A)=A:\nabla\xi(A)-\xi(A).
\label{eq:chen-path-functional}
\end{align}
Here, our Gaussian process of interest is $\beta Y_{n,p}$ defined as in~\eqref{eq:y_cov}. Therefore we have the overlap function $\xi(A)= \beta^2 \xi_p(\tr(A))$. Hence,
\begin{align}
\nabla_A\{\xi(A)\}
&=\beta^2\xi_p'(\Tr A)I_3, \qquad 
\theta(A)
=\beta^2\bigl\{(\Tr A)\xi_p'(\Tr A)-\xi_p(\Tr A)\bigr\}.
\label{eq:matrix-theta}
\end{align}
Plugging this into~\eqref{eq:chen-path-functional}, we obtain the simplified functional 
\begin{align}
\mathscr P_{p,\beta}(\pi)
:=&\ \E\log\iint
\exp\left\{w^{\rho_\pi}(\alpha)\cdot s
-\frac{\beta^2}{2}\xi_p'\bigl(\tr\pi(1)\bigr)\|s\|^2\right\}
\nu(\dd s)\mathcal R(\dd\alpha)
\nonumber\\
&\quad+\frac{\beta^2}{2}\int_0^1
\left\{(\tr\pi(u))\xi_p'(\tr\pi(u))-\xi_p(\tr\pi(u))\right\}\dd u,
\label{eq:matrix-Parisi-functional}
\end{align}
where $\rho_\pi(u):=\beta^2\xi_p'(\Tr\pi(u))I_3$. 
Now we state the free energy limit of $Y_{n,p}$.

\begin{lemma}\label{lem:matrix-formula}
For every fixed $p\ge2$ and every $\beta>0$, the limit $A_p(\beta):=\lim_{n\to\infty}A^Y_{n,p}(\beta)$
exists and satisfies
\begin{align}
A_p(\beta)=\frac{\beta^2}{2}+\inf_{\pi\in\Pi}\mathscr P_{p,\beta}(\pi).
\label{eq:matrix-Parisi-formula}
\end{align}
Here, $P_{p,\beta}(\pi)$ is defined as in~\eqref{eq:matrix-Parisi-functional}.
\end{lemma}

\begin{proof}
To invoke \cite[Corollary 8.3]{chen2025free}, we verify the assumptions 
for the covariance map $\xi$ defined by \eqref{eq:define_xi}.

(i) The single-spin prior is $P_1=\nu$, the uniform probability measure on $\Sph$.  Hence $P_1$ is supported on the unit ball of $\R^3$. 

(ii) The map $A\mapsto\xi(A)=\beta^2\xi_p(\Tr A)$ is a polynomial in the entries of $A$, so it has an absolutely convergent power
series and a locally Lipschitz gradient. It vanishes at $A=0$, and $\xi(A)=\xi(A^\dagger)$. $\xi(A) \ge 0$, for $A \in \Splus$. 

(iii)  If $A\succeq B\succeq0$, then $\xi(A)\ge \xi(B)$ and $\nabla \xi(A) \succeq \nabla \xi(B)$. To see this, note that $\Tr A\ge\Tr B\ge0$. Since $q\mapsto q^p$ and $q\mapsto pq^{p-1}$ are nondecreasing on $[0,\infty)$ we have  $\beta^2\xi_p(\Tr A)\ge\beta^2\xi_p(\Tr B)$, and $$\nabla \xi(A)= \beta^2\xi'_p(\Tr A) I_3 \succeq \beta^2\xi'_p(\Tr B) I_3= \nabla \xi(B).$$

(iv) For $A,B \in \Splus$ and $\lambda \in [0,1]$, convexity of $q\mapsto q^p$ on $[0,\infty)$ yields $\xi(\lambda A+ (1-\lambda)B) \le \lambda \xi(A)+ (1-\lambda) \xi(B).$
Hence, $\xi$ is convex on $\Splus$.

\noindent Thus~\cite[Corollary 8.3]{chen2025free} holds for self-overlap corrected Hamiltonian
\begin{align}
\widehat H_{n,p,\beta}(s)
:=\beta Y_{n,p}(s)-\frac n2\beta^2\xi_p(\Tr Q(s)),
\label{eq:corrected-H}
\end{align}
where $Q(s)$ is defined by~\eqref{eq:overlap}. Since every $s_i\in\Sph$, we have
\begin{align*}
\Tr Q(s)=\frac1n\sum_{i=1}^n \|s_i\|^2=1
\end{align*}
implying $$\widehat H_{n,p,\beta}(s)=\beta Y_{n,p}(s)-\frac{n\beta^2}{2}.$$
Hence we have 
that
\begin{align*}
\lim_{n\to\infty}\frac1n\E\log\int e^{\widehat H_{n,p,\beta}(s)}\nu^{\otimes n}(\dd s)
=\inf_{\pi\in\Pi}\mathscr P_{p,\beta}(\pi).
\end{align*}
Adding back self-overlap term $\beta^2/2$ proves the formula \eqref{eq:matrix-Parisi-formula}.
\end{proof}

The variational formula in Lemma~\ref{lem:matrix-formula} gives the required finite-temperature limit, but the optimization is on a path of $3\times3$ matrices. We next show that the variational problem can be simplified for our setup. To this end, define 
\begin{align}
\cM:=\{\alpha:[0,1]\to[0,1]:\alpha\text{ is nondecreasing and right-continuous, }\alpha(1)=1\}.
\label{eq:M}
\end{align}
For $\alpha\in\cM$, let $\Psi_{p,\beta,\alpha}$ denote the solution of
\begin{align}
\partial_q\Psi_{p,\beta,\alpha}
+\frac12\xi_p''(q)
\left\{\Delta\Psi_{p,\beta,\alpha}
+\beta\alpha(q)\norm{\nabla\Psi_{p,\beta,\alpha}}^2\right\}=0,
\label{eq:scalar-PDE}
\end{align}
with boundary condition
\begin{align}
\Psi_{p,\beta,\alpha}(1,x)=g_\beta(\norm{x}),
\qquad
g_\beta(r):=\frac1\beta\log\frac{\sinh(\beta r)}{\beta r},
\qquad g_\beta(0):=0.
\label{eq:g-beta}
\end{align}
Its value is given by the stochastic representation~\eqref{eq:finite-temp-control}.
\begin{align}
\cP_{p,\beta}(\alpha)
:=\Psi_{p,\beta,\alpha}(0,0)
-\frac\beta2\int_0^1q\xi_p''(q)\alpha(q)\dd q.
\label{eq:scalar-functional}
\end{align}

\begin{lemma}\label{lem:scalar-formula}
For every fixed $p\ge2$ and every $\beta>0$,
\begin{align}
\frac{A_p(\beta)}{\beta}
=\inf_{\alpha\in\cM}\cP_{p,\beta}(\alpha).
\label{eq:scalar-Parisi-formula}
\end{align}
Further, the minimizer in RHS is unique, and will be denoted by $\alpha_{p,\beta}$.
\end{lemma}

\begin{proof}
We first identify the endpoint of the matrix path in Lemma~\ref{lem:matrix-formula}.  For a symmetric $3 \times 3$ matrix $A$, let
\begin{align}
\widehat A_{p,\beta}(A)
:=\lim_{n\to\infty}\frac1n\E\log\int
\exp\left\{\widehat H_{n,p,\beta}(s)+n\,A:Q(s)\right\}
\nu^{\otimes n}(\dd s),
\label{eq:external-field-pressure}
\end{align}
where $Q(s)$ is defined by~\eqref{eq:overlap}. The limit exists for every symmetric $A$ by the same argument used in Lemma~\ref{lem:matrix-formula}. 
By~\cite[Proposition 2.3]{chen2024self}, we obtain that $A\mapsto\widehat A_{p,\beta}(A)$ is differentiable at $A=0$.  Define
\begin{align}
Q_{p,\beta}^\star:=\nabla_A\widehat A_{p,\beta}(0),
\label{eq:Qstar-def}
\end{align}
and the infimum may be restricted to paths satisfying $\pi(1)=Q_{p,\beta}^\star$.
Now, we evaluate the RHS. The model is invariant in distribution under simultaneous orthogonal rotation. Indeed, for $O\in \mathcal{O}(3)$, simultaneous rotation $s_i\mapsto Os_i$, both the base measure $\nu$ and overlap $q(s,t)$ remain unchanged, implying 
\begin{align*}
\widehat A_{p,\beta}(OAO^\top)=\widehat A_{p,\beta}(A).
\end{align*}
Differentiating at $A=0$ yields
$OQ_{p,\beta}^\star O^\top=Q_{p,\beta}^\star$ for every $O\in \mathcal{O}(3)$, so $Q_{p,\beta}^\star=aI_3$.  
As $\tr(Q(s))=1$, plugging $A=tI_3$ into~\eqref{eq:external-field-pressure} yields
\begin{align*}
\widehat A_{p,\beta}(tI_3)=\widehat A_{p,\beta}(0)+t.
\end{align*}
Differentiating at $t=0$ gives $1=I_3:Q_{p,\beta}^\star=\Tr Q_{p,\beta}^\star=3a$. Thus $Q_{p,\beta}^\star=\frac13I_3$.
Hence, we may restrict ourselves to paths with $\pi(1)=\frac13I_3$. Given $\pi\in \Pi$ with $\pi(1)=\frac13I_3$, define
\begin{align}
\bar\pi(u):=\frac{\Tr\pi(u)}3I_3.
\label{eq:trace-projection}
\end{align}
The path $\bar\pi$ is increasing and has the same endpoint as $\pi$. We claim the following: 
\begin{align}
\mathscr P_{p,\beta}(\bar\pi)=\mathscr P_{p,\beta}(\pi).
\label{eq:trace-equality}
\end{align}
This is immediate from the definition~\eqref{eq:matrix-Parisi-functional} since $\tr\bar\pi(u)=\tr\pi(u)$ for every $u$, which further implies $\rho_{\bar\pi}(u)=\rho_\pi(u)$. Therefore, \eqref{eq:trace-equality} shows we can replace $\pi$ with $\bar \pi$. 

Recall the definition of $\mathcal{M}$ in~\eqref{eq:M} and define, for any $\alpha \in \cM$, $\alpha^{-1}(u):= \inf\{q\in [0,1]: \alpha(q)\ge u\}$. Set $\mathcal{M}_d \subseteq \mathcal{M}$ denote all discrete cumulative distribution function on $[0,1]$ with a positive atom at $1$. Define $\Psi_{\star}(q)= \frac{q}3 I_3$. For any discrete path $\pi$ with $\pi(u)=c(u)I_3$, the function $u \mapsto \tr(\pi(u))$ is nondecreasing (as $\pi$ is nondecreasing in $\Splus$), and satisfies $\tr(\pi(1))=1$. Hence, the map equals $\alpha^{-1}$ for some $\alpha \in \mathcal{M}_d$, with $\pi= \Psi_{\star} \circ \alpha^{-1}$. 
Therefore, by \cite[Lemma 4.1]{chen2025parisi},
\begin{align}
F(\Psi_\star,\alpha)
=\mathscr P_{p,\beta}(\Psi_\star\circ\alpha^{-1}),
\label{eq:fixed-path-identity}
\end{align}
for $\alpha \in \mathcal{M}_d$ and $F$ is defined as in~\cite[(1.16)]{chen2025parisi}.  By a standard approximation argument from $\cM_d$ to $\cM$, we have 
\begin{align}
\inf_{\alpha\in\cM}F(\Psi_\star,\alpha)
=\inf_{\alpha\in\cM_{\rm d}}F(\Psi_\star,\alpha)= \inf_{\alpha \in \mathcal{M}_d}\mathscr P_{p,\beta}(\Psi_\star\circ \alpha^{-1})= \inf_{\alpha \in \mathcal{M}}\mathscr P_{p,\beta}(\Psi_\star\circ \alpha^{-1}),
\label{eq:discrete-alpha-inf}
\end{align}


Therefore, it is enough to compute $F(\Psi_\star,\alpha)$ for our Hamiltonian. To this end, define
$\mu(q)= \beta^2\xi_p'(q)I_3.$ Let
\(\Phi_\alpha\) denote the solution of the Parisi PDE~\cite[(1.12)]{chen2025parisi}:
\begin{align}
\partial_q \Phi_\alpha(q,x)
+\frac12 \dot\mu(q):
\left(
\nabla^2\Phi_\alpha(q,x)
+\alpha(q)\nabla\Phi_\alpha(q,x)
\nabla\Phi_\alpha(q,x)^\top
\right)=0.
\label{eq:chen-pde-specialization}
\end{align}
Since $\dot\mu(q)=\beta^2\xi_p''(q)I_3$, we obtain $I_3:
\nabla^2\Phi_\alpha=\Delta\Phi_\alpha$ and $I_3:
\bigl(\nabla\Phi_\alpha\nabla\Phi_\alpha^\top\bigr)
=\|\nabla\Phi_\alpha\|^2$. Therefore,~\eqref{eq:chen-pde-specialization} simplifies to
\begin{align}
\partial_q\Phi_\alpha
+\frac{\beta^2}{2}\xi_p''(q)
\left(
\Delta\Phi_\alpha
+\alpha(q)\|\nabla\Phi_\alpha\|^2
\right)=0.
\label{eq:Phi-pde}
\end{align}
The terminal condition is given by~\cite[(1.13)]{chen2025parisi}:
\begin{align}
    \Phi_\alpha(1,x) &= \log\int_{\mathbb{S}^2} \exp(x.s -\frac{1}{2}\mu(1):ss^\top) \nu(ds)\nonumber \\
    &=\log\int_{\mathbb{S}^2} \exp(x.s -\frac{1}{2}p \beta^2) \nu(ds)\nonumber \\
    &= -\frac{1}{2}p \beta^2+ \log\frac{\sinh\|x\|}{\|x\|},\label{eq:Phi-boundary}
\end{align}
defined continuously at $x=0$. Here, the second equality follows from $\int_{\mathbb S^2} e^{x.s}\nu(ds)= \frac{1}{2}\int_{-1}^1 e^{t\|x\|}dt$ by rotating $x$. Recall the definition of $\theta(A)$ from~\eqref{eq:matrix-theta}. At \(A=\Psi_\star(1)=I_3/3\),
\begin{align}
\theta(\Psi_\star(1))
&=
\frac13 I_3:
\bigl(\beta^2\xi_p'(1)I_3\bigr)
-\beta^2\xi_p(1) =
(p-1)\beta^2.
\label{eq:theta-endpoint}
\end{align}
Moreover,
\begin{align}
\Psi_\star(q):\dot\mu(q)
&=
\frac q3 I_3:
\bigl(\beta^2\xi_p''(q)I_3\bigr) =
\beta^2 q\xi_p''(q).
\label{eq:path-mu-term}
\end{align}
Combining the above two displays, we obtain that the definition of $F$ from~\cite[(1.16)]{chen2025parisi} simplifies to:
\begin{align}
F(\Psi_\star,\alpha)
=
\Phi_\alpha(0,0)
+\frac{(p-1)\beta^2}{2}
-\frac{\beta^2}{2}
\int_0^1
q\xi_p''(q)\alpha(q)\,dq.
\label{eq:F-fixed-path}
\end{align}
Defining
$$
\overline\Phi_\alpha(q,x):= \Phi_\alpha(q,x)+\frac{p\beta^2}{2},$$
the function
\(\overline\Phi_\alpha\) satisfies the same vector Parisi-PDE
\eqref{eq:Phi-pde}, with boundary condition 
$\overline\Phi_\alpha(1,x)
= \log\frac{\sinh\|x\|}{\|x\|}.$
Plugging $\Phi_\alpha(0,0) = \overline\Phi_\alpha(0,0) -\frac{p\beta^2}{2}$
into \eqref{eq:F-fixed-path}, we obtain
\begin{align}
A_p(\beta)
=
\inf_{\alpha\in\mathcal M}
\left\{
\overline\Phi_\alpha(0,0)
-\frac{\beta^2}{2}
\int_0^1
q\xi_p''(q)\alpha(q)\,dq
\right\}.
\label{eq:Ap-before-rescaling}
\end{align}
Setting $\Psi_{p,\beta,\alpha}(q,x)
:=
\frac1\beta
\overline\Phi_\alpha(q,\beta x)$ proves~\eqref{eq:scalar-Parisi-formula}.

Finally, since $\dot\mu(q)=\beta^2\xi_p''(q)I_3$ is positive definite for Lebesgue-a.e. \(q\in[0,1]\), by~\cite[Theorem 1.2]{chen2025parisi}, we obtain strict convexity of
\(\alpha\mapsto F(\Psi_\star,\alpha)\), and hence uniqueness of the
minimizer in \eqref{eq:scalar-Parisi-formula}.
\end{proof}

Now, we will analyze the case $\beta \rightarrow \infty$. For the remainder of the proof, we denote its unique minimizer of the RHS of \eqref{eq:scalar-Parisi-formula} by $\alpha_{p,\beta}$. Our argument is a generalization of \cite{auffinger2015parisi}. Formally, we will prove
$$
\lim_{\beta\to\infty}\frac{A_p(\beta)}{\beta}
= \inf_{\gamma\in\mathcal U}\mathcal P_p(\gamma),
$$
where $\mathcal P_p$ is defined by~\eqref{eq:define_variational}.

\noindent \textbf{Upper bound:} We first prove the upper bound. The stochastic representation corresponding to the PDE~\eqref{eq:scalar-PDE} is the \(\mathbb R^3\)-valued version of \cite[Theorem 2 and Corollary 1]{auffinger2015parisi}. Recall that
$\dot\mu(q)=\beta^2\xi_p''(q)I_3$.  Applying
\cite[Proposition~3.2]{chen2025parisi} and using the rescaling
$\Psi_{p,\beta,\alpha}(q,x)=\beta^{-1}\overline\Phi_\alpha(q,\beta x)$ from the proof of Lemma~\ref{lem:scalar-formula}, we have
\begin{align}
\Psi_{p,\beta,\alpha}(s,x)
=
\sup_{u\in D[s,1]}
\Bigg\{
&\mathbb E\,g_\beta\left(
\left\|
x+\int_s^1\sqrt{\xi_p''(q)}\,dB_q
+\beta\int_s^1\xi_p''(q)\alpha(q)u(q)\,dq
\right\|
\right) \nonumber\\
&\qquad
-\frac{\beta}{2}\int_s^1
\xi_p''(q)\alpha(q)\mathbb E\|u(q)\|^2\,dq
\Bigg\},
\label{eq:finite-temp-control}
\end{align}
where recall that $D[s,t]$ the collection of all progressive measurable processes $u:[s,t]\mapsto \mathbb{R}^3$ satisfying $\sup_{s\le r \le t} \|u(r)\| \le 1$. Adding the constraint $\sup_{0\le t \le 1} \|u(t)\| \le 1$ can be justified as follows: the terminal condition $x \mapsto g_{\beta}(\|x\|)$ is 1-Lipschitz, implying $\|\nabla \Psi_{p,\beta,\alpha}\| \le 1$. By the stochastic control representation of Parisi PDE~\cite[Proposition 3.2]{chen2025parisi}, and the maximizer $u$ equals $\nabla \Psi_{p,\beta,\alpha}(t,x)$ for some $(t,x)$.

First fix a bounded $\gamma\in\mathcal U$.  For
$\beta\ge\|\gamma\|_\infty$, define
$\alpha_{\beta}(q)=\gamma(q)/\beta$ for $q<1$ and set $\alpha_{\beta}(1)=1$.  Then
$\alpha_{\beta} \in\mathcal M$. Plugging $\beta \alpha_{\beta}=\gamma$ into~\eqref{eq:finite-temp-control}, we obtain 
\begin{align}
\Psi_{p,\beta,\alpha_\beta}(0,0)
=
\sup_{u\in D[0,1]}
\Bigg\{
&\mathbb E\,g_\beta\left(
\left\|\int_0^1\sqrt{\xi_p''(q)}\,dB_q
+\int_0^1\xi_p''(q)\gamma(q)u(q)\,dq
\right\|
\right) \nonumber\\
&\qquad
-\frac{1}{2}\int_0^1
\xi_p''(q)\gamma(q)\mathbb E\|u(q)\|^2\,dq
\Bigg\},
\label{eq:finite-temp-control-at-zero}
\end{align}
We note that, compared to~\eqref{eq:define_p}, the above display has $g_\beta(\|.\|)$ instead of $\|.\|$. Fix $u \in D[0,1]$ and define $X^u
= \int_0^1\sqrt{\xi_p''(q)}\,dB_q +\int_0^1\xi_p''(q)\gamma(q)u(q)\,dq$. Since $\|u(q)\|\le 1$, we have 
\begin{align*}
    \mathbb{E}\|X^u\|\le \mathbb E\|\int_0^1\sqrt{\xi_p''(q)}\,dB_q \|+ \int_0^1\xi_p''(q)|\gamma(q)| \,dq \le \sqrt{3p}+ p(p-1)\|\gamma\|_1:= \kappa
\end{align*}
Hence, $\sup_{u}\mathbb E\|X^u\| \le \kappa$. Next, note that for $t \ge 0$, 
$$0 \le t-g_{\beta}(t) \le \frac{1}{\beta} \log (2(1+\beta t)),$$
which implies using Jensen's inequality that
$$0 \le \E \|X^u\|- \E g_{\beta}(\|X^u\|) \le \frac{1}{\beta} \log (2(1+\beta \kappa)).$$
Taking $\beta \rightarrow \infty$, we obtain $\Psi_{p,\beta,\alpha_\beta}(0,0) \rightarrow \Psi_{p,\gamma}(0,0)$.
Hence, using Lemma~\ref{lem:scalar-formula}, we obtain
\begin{align*}
    \limsup_{\beta\rightarrow \infty} \frac{A_p(\beta)}{\beta} \le \Psi_{p,\gamma}(0,0)- \frac{1}{2} \int_{0}^1 t \xi''_p(t) \gamma(t)dt= \mathcal{P}_p(\gamma).
\end{align*}

Finally, if $\gamma$ is unbounded, define $\gamma_M=\min(\gamma,M)$, where $\gamma_M\in \mathcal{U}$, bounded, and $\|\gamma_M-\gamma\|_1\rightarrow 0$. It is easy to see 
$$|\mathcal{P}_p(\gamma)- \mathcal{P}_p(\gamma_M)| \le C_p \|\gamma_M-\gamma\|_1\rightarrow 0,$$ for some $C_p>0$ completing the proof of the upper bound.

\noindent \textbf{Lower bound:}
For the lower bound, set $\gamma_{p,\beta}(q):=\beta\alpha_{p,\beta}(q)$, $\xi_\beta(q):=\beta^2\xi_p(q)=\beta^2q^p$ and $\theta_\beta(q):=q\xi_\beta'(q)-\xi_\beta(q)=(p-1)\beta^2q^p$. Define $h(x):=\log\int_{\mathbb{S}^2}e^{x\cdot s}\nu(ds)$.  We begin by proving differentiability of $Q_{p,\beta}$, defined below.

A distribution function with at most $k$ atoms can be written as
$$
\alpha(q)=\sum_{\ell=0}^k m_\ell\mathbf 1_{[q_\ell,q_{\ell+1})}(q),
$$
where $0=m_0\le m_1\le\cdots\le m_k=1$ and $0=q_0\le q_1\le\cdots\le q_k\le q_{k+1}=1$. Let $z_0,\ldots,z_k$ be independent centered Gaussian vectors in $\R^3$ with
$$
\E z_\ell z_\ell^\top=\{\xi_\beta'(q_{\ell+1})-\xi_\beta'(q_\ell)\}I_3.
$$
and set $X_k=h(\sum_{\ell=0}^kz_\ell)$.  Recursively, for $\ell=k,\ldots,1$,
$$
X_{\ell-1}=
\begin{cases}
 m_\ell^{-1}\log\E_\ell e^{m_\ell X_\ell},&m_\ell>0,\\
 \E_\ell X_\ell,&m_\ell=0,
\end{cases}
$$
where $\E_\ell$ denotes the expectation in $z_j$, $\ell \le j \le k$.  Define
\begin{align}
Q_{p,\beta}(\alpha)
:=\E X_0-\frac12\sum_{\ell=1}^k m_\ell\{\theta_\beta(q_{\ell+1})-\theta_\beta(q_\ell)\}.
\label{eq:finite-step-Q}
\end{align}
Similar notation is common in spin-glass literature and is used to prove differentiability of Parisi functional~\cite{panchenko2008differentiability}. For discrete $\alpha$, we have $Q_{p,\beta}(\alpha)=\beta\mathcal{P}_{p,\beta}(\alpha)$. If $Z\sim N(0,\tau I_3)$, then
\begin{align}
\E_Ze^{h(x+Z)}
&=\int_{S^2}e^{x\cdot s}\E e^{Z\cdot s}\nu(ds)
=e^{\tau/2}e^{h(x)},
\label{eq:fixed-length}
\end{align}
because $\|s\|=1$. Therefore, the precise recursion of~\cite[(1.9)]{panchenko2008differentiability} holds.
This implies, by~\cite[Lemma 1]{panchenko2008differentiability} the following: If $\alpha$ minimizes $Q_{p,\beta}$ over distribution functions with at most $k$ atoms, then
$$
\partial_\beta Q_{p,\beta}(\alpha)
=\beta\left(1-\int_0^1q^p\,d\alpha(q)\right).
$$

Next, we have to consider the minimizer of $Q_{p,\beta}$ for all probability distribution $\alpha$. The argument is exactly same as~\cite[Theorem 1]{panchenko2008differentiability} once we can show 
\begin{equation}\label{eq:second_derivative}
    |Q_{p,\beta+y}(\alpha)- Q_{p,\beta}(\alpha) - y\partial_\beta Q_{p,\beta}(\alpha)| \le C y^2,
\end{equation} where $\beta,\beta+y \in [0,B]$ and $C=C(p,B)$ does not depend on $\alpha$.

\noindent \textbf{Proof of~\eqref{eq:second_derivative}:}
For discrete $\alpha$, define
\begin{align*}
    r_{\beta,\alpha}(u):= \frac12\xi_\beta'\bigl(\alpha^{-1}(u)\bigr)I_3 = \beta^2 r_{1,\alpha}(u),
\qquad
r_{1,\alpha}(u) = \frac12\xi_p'\bigl(\alpha^{-1}(u)\bigr)I_3.
\end{align*}
By the standard cascade computation
\cite[(3.7)]{chen2025free}, 
the recursion above gives
\begin{equation}
    \mathbb E X_0= -\psi(r_{\beta,\alpha})+ \frac{1}{2}\xi_\beta'(1)= -\psi(r_{\beta,\alpha})+  \frac{1}{2}p\beta^2,
\end{equation}
where $\psi(r)$ is defined as \cite[(5.19)]{chen2025free}. Hence,
\begin{equation}\label{eq:Q-cascade}
Q_{p,\beta}(\alpha)
=
-\psi(r_{\beta,\alpha})
+
\frac{\beta^2}{2}
\left\{
p-\int_0^1 q\xi_p''(q)\alpha(q)\,dq
\right\}.
\end{equation}
 Since $0\le \alpha^{-1}\le1$, we have uniformly in discrete $\alpha$,
$$
\|r_{1,\alpha}\|_{L^1} \le \frac{\sqrt3\,p}{2},
\qquad
\|r_{1,\alpha}\|_{L^2}^2 \le \frac{3p^2}{4}.
$$
Using \cite[(5.15),(5.16)]{chen2025free}, with $N=1$ and $t=0$, we obtain
$$
\left| \psi(r')-\psi(r) -\langle \partial_q\psi(r),r'-r\rangle_{L^2}\right| \le 8\|r'-r\|_{L^2}^2, \qquad
\left| \langle \partial_q\psi(r),v\rangle_{L^2}\right|
\le \|v\|_{L^1}.$$
Taking $r=r_{\beta,\alpha}$ and $r'=r_{\beta+y,\alpha}$, we have $r_{\beta+y,\alpha}-r_{\beta,\alpha}
=
(2\beta y+y^2)r_{1,\alpha}$.
Hence, for $0<\beta,\beta+y\le B$,
\begin{align}\label{eq:Q-quadratic}
&\left| Q_{p,\beta+y}(\alpha)-Q_{p,\beta}(\alpha)
-y\partial_\beta Q_{p,\beta}(\alpha) \right| \\
&\le \left|
\psi(r')-\psi(r)
-\langle \partial_q\psi(r),r'-r\rangle_{L^2}
\right|+ y^2\left|
\langle \partial_q\psi(r),r_{1,\alpha}\rangle_{L^2}
\right|+ \frac{y^2}{2}\Big|p-\int_0^1 q\xi_p''(q)\alpha(q)\,dq\Big|\\
&\le 8(2\beta y+y^2)^2 \|r_{1,\alpha}\|^2_{L^2}+y^2\|r_{1,\alpha}\|_{L^1}+\frac{p}{2}y^2 \le C_{B,p}y^2
\end{align}
uniformly over discrete $\alpha$ with $C_{p,B}=\frac{(1+\sqrt3)p}{2}+24B^2p^2$.
Therefore, using~\cite[Theorem 1]{panchenko2008differentiability},
\begin{align}
A_p'(\beta)
=\beta\left(1-\int_0^1q^p\,d\alpha_{p,\beta}(q)\right)
=p\int_0^1q^{p-1}\gamma_{p,\beta}(q)dq.
\label{eq:derivative}
\end{align}

Our upper bound of $A_p(\beta)/\beta$ together with convexity bounds the RHS of the above display. Indeed, choose $\beta_0,C$ such that $A_p(\beta)/\beta\le C$ for all $\beta\ge\beta_0$. Since $A_p \ge 0$ and convex, we have, for all $\beta\ge \beta_0$,
$$A_p'(\beta)\le \frac{A_p(2\beta)-A_p(\beta)}{\beta} \le \frac{A_p(2\beta)}{\beta} \le 2C.$$
Hence, \eqref{eq:derivative} implies that for $\beta \ge \beta_0$,
\begin{align}
p\int_0^1q^{p-1}\gamma_{p,\beta}(q)dq\le C.
\label{eq:weighted-bound}
\end{align}
Since $\gamma_{p,\beta}$ is nondecreasing, 
\begin{align*}
    \int_0^1\gamma_{p,\beta}(q)dq &=  \int_0^{1/2}\gamma_{p,\beta}(q)dq+  \int_{1/2}^1\gamma_{p,\beta}(q)dq\\
    &\le \frac12 \gamma_{p,\beta}\Big(\frac12\Big) + \frac{2^{p-1}}{p} \int_{1/2}^1 pq^{p-1} \gamma_{p,\beta}(q)dq \le \frac{C}{(1-2^{-p})} + \frac{2^{p-1}C}{p}=:C_p
\end{align*}
Hence, 
\begin{align}
\sup_{\beta\ge\beta_0}\int_0^1\gamma_{p,\beta}(q)dq<\infty.
\label{eq:L1-bound}
\end{align}
The same argument also yields $\gamma_{p,\beta}(t) \le C_p/(1-t)$. Thus, along any sequence $\beta \rightarrow \infty$, Helly's theorem gives a subsequence and some $\gamma \in \mathcal{U}$ such that $\gamma_{p,\beta}(q) \rightarrow \gamma(q)$ for every continuity point $q<1$. Since $\xi_p'' \le p(p-1)$, we obtain by~\eqref{eq:L1-bound}, $\sup_{\beta\ge\beta_0}\int_0^1\xi_p''(q)\gamma_{p,\beta}(q)dq<\infty.$ Hence along a further subsequence,  $\xi_p''(q)\gamma_{p,\beta}(q)dq$ converge weakly,
\begin{align}
\xi_p''(q)\gamma_{p,\beta}(q) dq \Longrightarrow \xi_p''(q)\gamma(q)dq+ \upsilon \delta_1
\label{eq:measure-limit}
\end{align}
for some $\upsilon \ge0$. So,
\begin{align}
\frac12\int_0^1q\,\xi_p''(q)\gamma_{p,\beta}(q) dq
\longrightarrow
\frac12\int_0^1q\xi_p''(q)\gamma(q)dq+\frac{\upsilon}2.
\label{eq:penalty-limit}
\end{align}
Next, we prove a lower bound.

\begin{lemma}
\label{prop:endpoint-gain} With $\upsilon$ defined as~\eqref{eq:measure-limit},
$$
\liminf_{\beta\to\infty}\Psi_{p,\beta,\alpha_{p,\beta}}(0,0)
\ge \Psi_{p,\gamma}(0,0)+\frac \upsilon 2.
$$
Here $\Psi_{p,\beta,\alpha}$ and $\Psi_{p,\gamma}$ are defined by~\eqref{eq:scalar-PDE} and~\eqref{eq:define_p} respectively.
\end{lemma}

\begin{proof}
Define a finite measure $\nu_0$ on
$[0,1]$ by
$$
\nu_0(ds)
=
\gamma(s)\mathbf 1_{[0,1)}(s)\,ds
+
\frac{\upsilon}{\xi_p''(1)}\delta_1(ds).
$$
Then
$\xi_p''(s)\nu_0(ds)
=
\xi_p''(s)\gamma(s)\,ds+\upsilon\delta_1(ds)$. Fix a unit vector $e\in\mathbb R^3$.  For $n\ge1$, define
$$
g_n(x)
=
\frac{x+n^{-1}e}
{\max\{n^{-1},\|x+n^{-1}e\|\}},
\qquad x\in\mathbb R^3.
$$
Then $\|g_n(x)\|\le1$ and
$$
g_n(x)\longrightarrow
g(x):=
\begin{cases}
x/\|x\|, & x\neq0,\\
e, & x=0.
\end{cases}
$$
Fix $u\in D[0,1]$.  For $\varepsilon\in(0,1)$, define
\begin{align}
\phi_{\varepsilon,n}(s)
:={}&
u(s)\mathbf 1_{[0,\varepsilon)}(s) \nonumber\\
&+
g_n\left(
\int_0^s\xi_p''(r)\gamma(r)u(r)\,dr
+
\int_0^s\sqrt{\xi_p''(r)}\,dB_r
\right)
\mathbf 1_{[\varepsilon,1]}(s).
\label{eq:phi-epsilon-n}
\end{align}
The process $\phi_{\varepsilon,n}$ is progressively measurable and
$\|\phi_{\varepsilon,n}(s)\|\le1$, so
$\phi_{\varepsilon,n}\in D[0,1]$. 
For fixed $\varepsilon,n$, we have by~\eqref{eq:measure-limit}:
\begin{align*}
&\int_{0}^1\xi''_p(s)\gamma_{p,\beta}(s)\phi_{\varepsilon,n}(s)ds \xrightarrow{L^1}\int_{0}^1\xi''_p(s)\phi_{\varepsilon,n}(s)\nu_0(ds),\\
&\int_{0}^1\xi''_p(s)\gamma_{p,\beta}(s)\E\|\phi_{\varepsilon,n}(s)\|^2ds \xrightarrow{L^1}\int_{0}^1\xi''_p(s)\E\|\phi_{\varepsilon,n}(s)\|^2\nu_0(ds).\\
\end{align*}
Recall that for $t \ge 0$, $0 \le t-g_{\beta}(t) \le \frac{1}{\beta} \log (2(1+\beta t))$. This implies
\begin{align}
\liminf_{\beta\to\infty}
\Psi_{p,\beta,\alpha_{p,\beta}}(0,0)
\ge{}&
\mathbb E\left\|
\int_0^1
\xi_p''(s)\phi_{\varepsilon,n}(s)\nu_0(ds)
+
\int_0^1\sqrt{\xi_p''(s)}\,dB_s
\right\| \nonumber\\
&-
\frac12
\int_0^1
\xi_p''(s)
\mathbb E\|\phi_{\varepsilon,n}(s)\|^2
\nu_0(ds).
\label{eq:endpoint-epsilon}
\end{align}
Letting $\varepsilon\uparrow1$ and setting
$$
S
=
\int_0^1\xi_p''(s)\gamma(s)u(s)\,ds
+
\int_0^1\sqrt{\xi_p''(s)}\,dB_s,
$$
we have, by dominated convergence theorem,
\begin{align}
\liminf_{\beta\to\infty}
\Psi_{p,\beta,\alpha_{p,\beta}}(0,0)
\ge{}&
\mathbb E\|S+\upsilon g_n(S)\| \nonumber\\
&-
\frac12
\int_0^1
\xi_p''(s)\gamma(s)
\mathbb E\|u(s)\|^2\,ds
-
\frac{\upsilon}{2}\mathbb E\|g_n(S)\|^2.
\label{eq:endpoint-n}
\end{align}
Finally, let $n\to\infty$.  By the definition of $g$,
$$
\|S+\upsilon g(S)\|=\|S\|+\upsilon,
\qquad
\|g(S)\|=1
$$
for every $S\in\mathbb R^3$.  Hence
\begin{align}
\liminf_{\beta\to\infty}
\Psi_{p,\beta,\alpha_{p,\beta}}(0,0)
\ge{}&
\mathbb E\|S\|
-
\frac12
\int_0^1
\xi_p''(s)\gamma(s)
\mathbb E\|u(s)\|^2\,ds
+
\frac{\upsilon}{2}.
\end{align}
Taking the supremum over $u\in D[0,1]$ completes the proof.
\end{proof}

Combining Lemma~\ref{lem:scalar-formula} with Lemma~\ref{prop:endpoint-gain} and \eqref{eq:penalty-limit}, we obtain
\begin{align*}
    \liminf_{\beta\rightarrow\infty} \frac{A_p(\beta)}{\beta} &\ge \Psi_{p,\gamma}(0,0)+\frac{\upsilon}{2}-\frac12
\int_0^1
s \xi_p''(s)\gamma(s) ds- \frac{\upsilon}{2}\\
&= \mathcal{P}_p(\gamma) \ge \inf_{\tilde \gamma \in \mathcal{U}}\mathcal{P}_p(\tilde \gamma) 
\end{align*}
This completes the proof of the theorem.
\end{proof}

\begin{proof}[Proof of Proposition~\ref{cor:large_p}]
The proof is split into two parts.

\noindent \textbf{Upper bound:}

The upper bound can be proved by selecting constant function $\gamma(t) \equiv m$ for some $m>0$. For such constant functions, we have by definition~\eqref{eq:define_p},
\begin{equation*}
    \Psi_{p,\gamma}(0,0)= \sup_{u \in D[0,1]}\left\{\mathbb{E}\Bigg[\Big\|\int_{0}^1 \sqrt{\xi''_{p}(t)} dB_t+ m\int_{0}^1 \xi''_p(t) u(t) dt\Big\|\Bigg]- \frac{m}{2}\int_0^1\xi''_p(t) \mathbb{E}(\|u(t)\|^2) dt\right\}.
\end{equation*}
Set $r=\xi_p'(t)=pt^{p-1}$ and define
$$
W_r
:=
\int_0^{(r/p)^{1/(p-1)}}
\sqrt{\xi_p''(t)}\,dB_t,
\qquad 0\le r\le p.
$$
The process $W$ is a standard Brownian motion in $\mathbb R^3$~\cite{revuz2013continuous}. Therefore, reparametrization yields
\begin{align}
    \Psi_{p,\gamma}(0,0)
    &=\sup_{u \in D[0,p]}
    \left\{
    \mathbb E\left\|
    W_p+m\int_0^p u(r)\,dr
    \right\|
    -\frac m2
    \int_0^p\mathbb E\|u(r)\|^2\,dr
    \right\} \nonumber \\
    & \le \frac{1}{m} \log\mathbb E e^{m\|W_p\|},
    \label{eq:constant-control}
\end{align}
where the last inequality is Bou\'e-Dupuis variational formula~\cite{boue1998variational}, whose supremum is over all square-integrable progressively measurable processes \(u\). Since $W_p\stackrel{d}{=}\sqrt p\,Z$ for
$Z\sim N(0,I_3)$ and
$\int_0^1t\xi_p''(t)\,dt=p-1$,
\begin{align*}
    \bP_p(\gamma)
    \le
    \frac1m\log\mathbb E e^{m\sqrt p\,\|Z\|}
    -
    \frac{m(p-1)}2.
    \label{eq:constant-upper}
\end{align*}
Note that $\|Z\|$ has density
$\sqrt{2/\pi}\,z^2e^{-z^2/2}$ on $[0,\infty)$.  Hence, 
\begin{align*}
    \mathbb E e^{m\sqrt p \|Z\|}
    &=
    \sqrt{\frac2\pi}
    e^{m^2p/2}
    \int_0^\infty x^2e^{-(x-m\sqrt p)^2/2}\,dx \le
    C(1+m^2p)e^{m^2p/2},
\end{align*}
implying that 
\begin{align*}
    \bP_p(\gamma)
    \le
    \frac m2
    +
    \frac1m\log\{C(1+m^2p)\}.
\end{align*}
Taking $m=\sqrt{2\log p}$ yields the required upper bound. 

\noindent \textbf{Lower bound:}

Fix $\gamma\in\mathcal U$. Recall, that 
$|\mathcal{P}_{p}(\gamma) -\mathcal{P}_{p}(\tilde \gamma)| \le C_p \|\gamma-\tilde \gamma\|_1$, where $C_p>0$ is a constant independent of $\gamma$. Hence, choosing $\tilde \gamma=\gamma \wedge K$ with $K \rightarrow \infty$, we can restrict ourselves to bounded $\gamma$. 

For any such bounded $\gamma$, consider the SDE
\begin{equation}\label{eq:lower_bound_sde}
    dX_t = \sqrt{ \xi''_p(t)} dB_t+\xi''_p(t) \gamma(t) \frac{X_t}{\|X_t\|} dt, \qquad X_0=0,
\end{equation}
with the standard convention $x/\|x\|=0$ if $x=0$. With the reparametrization $r=\xi_p'(t)=pt^{p-1}$, existence of strong solution of the process is guaranteed by~\cite[Theorem 1]{veretennikov1980strong}. Next, as the drift is bounded, by Girsanov's Theorem, there exists a measure under which the process $\tilde B_t= B_t+ \int_0^t \sqrt{\xi''_p(s)} \gamma(s) \frac{X_s}{\|X_s\|} ds$ is standard Brownian motion and
\begin{equation}X_t= \int_0^t \sqrt{\xi''_p(s)} d \tilde B_s \overset{d}{=} W_{\xi'_p(t)}=W_r,\end{equation}
for standard Brownian motion in $\mathbb R^3$. Hence, $\|X_t\| > 0$ almost surely for $t>0$.

To obtain a lower bound, we choose $u$ such that $u(0)=0$, $u(t)= X_t/\|X_t\|$, $t>0$. Note that $u\in D[0,1]$ since $\|u(t)\|=1$ for all $t>0$, almost surely. Hence we obtain, from definition~\eqref{eq:define_p}
\begin{equation}\label{eq:bound_i}
    \Psi_{p,\gamma}(0,0) \ge \mathbb E \|X_1\| - \frac{1}{2} \int_{0}^1 \xi''_p(t) \gamma(t)dt.
\end{equation}
Set $R_t= \|X_t\|$. Under the Girsanov measure described above, we have $R_t$ is a time-changed three-dimensional Bessel process. By It\^o formula of Bessel process~\cite[Chapter XI]{revuz2013continuous}:
\begin{align}
    dR_t &= \frac{X_t}{R_t} dX_t+ \frac{\xi''_p(t)}{R_t} dt \nonumber \\
    &= \sqrt{\xi''_p(t)} \frac{X_t}{R_t} dB_t + \xi''_p(t) \gamma(t) dt+ \frac{\xi''_p(t)}{R_t} dt \nonumber \\
    &=: \sqrt{\xi''_p(t)} db_t + \xi''_p(t) \gamma(t) dt+ \frac{\xi''_p(t)}{R_t} dt, \label{eq:rt} 
\end{align}
with $b_t=\int_0^t \frac{X_s}{R_s} dB_s$. Since $b_t$ has quadratic variation $t$, it is a one-dimensional Brownian motion. Now taking expectation,
\begin{equation}\label{eq:Rt_expectation}
    \mathbb E R_1= \int_0^1\xi''_p(t) \gamma(t) dt+ \int_0^1 \xi''_p(t) \mathbb{E}\frac{1}{R_t} dt
\end{equation}
Plugging this into~\eqref{eq:bound_i} yields
\begin{equation}\label{eq:bound_ii}
    \mathcal{P}_p(\gamma) \ge \frac{1}{2}\int_0^1 (1 - t)\xi''_p(t) \gamma(t) dt+ \int_0^1 \xi''_p(t) \mathbb{E}\frac{1}{R_t} dt
\end{equation}
Define
\begin{equation}
    A(t):= \int_0^t \xi''_p(s) \gamma(s) ds, \qquad M:=\int_0^1 A(t)\,dt= \int_0^1 (1 - t)\xi''_p(t) \gamma(t) dt.
\end{equation}
Assume first $M>0$. By definition of $R_t$,
\begin{align*}
    \mathbb E R_t
    &\le
    \mathbb E\left\|
    \int_0^t\sqrt{\xi_p''(s)}\,dB_s
    \right\|+A(t) \le
    \sqrt{3\xi_p'(t)}+A(t),
\end{align*}
where the second inequality follows from It\^o isometry and Cauchy-Schwarz inequality. Hence, by Jensen's inequality,~\eqref{eq:bound_ii} gives
\begin{align}
    \bP_p(\gamma) \ge \frac{M}2 + \int_0^1 \xi''_p(t) \frac{1}{\mathbb{E} R_t} dt
    \ge
    \frac{M}{2}
    +
    \int_0^1
    \frac{\xi_p''(t)}
    {\sqrt{3\xi_p'(t)}+A(t)}\,dt.
\label{eq:deterministic-lower}
\end{align}
Since $\gamma$ is nondecreasing,
$A(t)\le\gamma(t)\xi_p'(t)$, implying 
\begin{align*}
    \frac{d}{dt}
    \Bigg[
    \frac{A(t)}{\xi_p'(t)}
    \Bigg]
    =
    \frac{\xi_p''(t)}
    {\xi_p'(t)^2}
    \left\{
    \gamma(t)\xi_p'(t)-A(t)
    \right\}
    \ge0
\end{align*}
for a.e.\ $t\in(0,1)$.  Hence,
\begin{align}
    M
    &\ge
    \int_t^1 A(s)\,ds
    \ge
    \frac{A(t)}{\xi_p'(t)}
    \int_t^1\xi_p'(s)\,ds
    =
    \frac{A(t)}{\xi_p'(t)}(1-t^p).
\label{eq:A-bound-general}
\end{align}
For the rest of the argument, we take $p$ sufficiently large. In particular, when $\xi_p'(t)\le p/\log p$, we have $t^p\le(\log p)^{-1}$ which implies
\begin{align}
    A(t)
    \le
    \frac{M}{1-(\log p)^{-1}}\xi_p'(t).
\label{eq:A-bound}
\end{align}
Using \eqref{eq:deterministic-lower} with $r=\xi_p'(t)$, we get
\begin{align}
    \bP_p(\gamma)
    &\ge
    \frac{M}{2}
    +
    \int_0^{p/\log p}
    \frac{dr}
    {\sqrt{3r}
    +Mr/\{1-(\log p)^{-1}\}} \nonumber \\
    &=\frac{M}{2}
    +
    \frac{2\{1-(\log p)^{-1}\}}{M}
    \log\left(
    1+
    \frac{M}
    {\sqrt{3}\{1-(\log p)^{-1}\}}
    \sqrt{\frac{p}{\log p}}
    \right).
\label{eq:M-integral}
\end{align}


If $M\ge3\sqrt{\log p}$, the first summand of the above display is $\ge \sqrt{2\log p}$, providing the lower bound. If
$0<M\le(\log p)^{-1}$, the second summand is $\ge \sqrt{2\log p}$, which again gives the lower bound.  Finally, if $(\log p)^{-1}\le M\le3\sqrt{\log p}$, observe that
\begin{align*}
    \log\left(
    1+ \frac{M}{\sqrt{3}\{1-(\log p)^{-1}\}}\sqrt{\frac{p}{\log p}}
    \right)\ge   \frac12\log p - \frac{3}{2}\log\log p -O(1).
\end{align*}
Hence,
\begin{align*}
    \bP_p(\gamma)
    &\ge
    \frac{M}{2}
    +(1-o(1))\frac{\log p}{M}\ge
    (1-o(1))\sqrt{2\log p},
\end{align*}
If $M=0$, then $\gamma=0$, and $\bP_p(\gamma)= \sqrt{p} \E \|Z\|$, yielding the required bound. This proves the lower bound, finishing the proof of the proposition. 
\end{proof}

\begin{proof}[Proof of Proposition~\ref{prop:universality}] 
   For $m_n:=3^pN_n$ pairs $j=(\bar \bi,\bar \ba)$, set
\begin{equation}\label{eq:define_varphi}
    \varphi_j(s)=\prod_{r=1}^p s_{i_r}^{a_r},
\end{equation} so $|\varphi_j(s)|\le1$ for any $s\in (\mathbb S^2)^n$. Using Bloch vectors as in~\eqref{eq:x_proc_reduce}, for deterministic $x\in\mathbb R^{m_n}$,
\begin{equation}\label{eq:field-free-energy}
X^x_{n,p}(s)=\sqrt{\frac n{N_n}}\sum_{j=1}^{m_n}x_j\varphi_j(s),
\qquad M_n(x)=\frac1n\max_{s\in(\Sph)^n}X^x_{n,p}(s)
\end{equation}
Define
$$F_{n,\beta}(x)=\frac1{n\beta}\log\int_{(\Sph)^n}
 e^{\beta X^x_{n,p}(s)}\,\nu^{\otimes n}(ds),\qquad\beta>0,$$
where $\nu$ is normalized surface measure on $\Sph$.
In particular, $M_n(x)=E^\star_{n,\mathrm{prod}}(p;x)$ and
$\E F_{n,\beta}(\alpha)=A^X_{n,p}(\beta)/\beta$ by~\eqref{eq:define_ax}.

For each $j=(\bar \bi,\bar \ba)$, we have that $i_1,\ldots,i_p$ are distinct, implying
\begin{equation}
\int_{(\Sph)^n}\varphi_j(s)\,\nu^{\otimes n}(ds)
=\prod_{r=1}^p\int_{\Sph}u^{a_r}\,\nu(du)=0.
\end{equation}
The last equality follows by symmetry. Thus $\int_{(\Sph)^n} X^x_{n,p}\,d\nu^{\otimes n}=0$ for every deterministic $x$. By Jensen's inequality, we have
$0\le F_{n,\beta}(x)\le M_n(x)$.

Fix $0<\delta<1$ and any $s=(s_1,\ldots,s_n)\in(\Sph)^n$. Define $C_i=\{u\in\Sph:u\cdot s_i\ge1-\delta\}$, implying $\nu(C_i)= \frac{\delta}{2}$. Define $T=(T_1,\ldots,T_n)$, where $T_i$'s are independent random variable, with law $\nu$ conditioned on $C_i$. Then $\mathbb E_T(T_i)= (1-\frac{\delta}{2}) s_i$. Hence,
\begin{equation}
\E_T X^x_{n,p}(T)
=\sqrt{\frac n{N_n}}\sum_{\bar \bi,\bar \ba}x[\bar \bi,\bar \ba]
 \prod_{r=1}^p\E_T T_{i_r}^{a_r}
=\Big(1-\frac{\delta}{2}\Big)^pX^x_{n,p}(s).
\end{equation}
Restricting the partition function to $\prod_{i=1}^n C_i$, we obtain
\begin{equation}
\int_{(\mathbb S^2)^n} e^{\beta X^x_{n,p}(u)}\,\nu^{\otimes n}(du)
\ge \left(\frac\delta2\right)^n\E_T e^{\beta X^x_{n,p}(T)}
\ge \left(\frac\delta2\right)^n e^{\beta (1-\frac{\delta}{2})^p X^x_{n,p}(s)}.
\end{equation}
Taking the supremum over $s$ achieves the bound
\begin{equation}\label{eq:cap-bound}
\Big(1-\frac{\delta}{2}\Big)^p M_n(x)-\frac1\beta\log\frac2\delta
\le F_{n,\beta}(x)\le M_n(x),
\qquad 0<\delta<1,\quad\beta>0.
\end{equation}

Equipped with this bound, we invoke the Lindeberg comparison principle~\cite{chatterjee2005simple}. To this end, let $\langle\cdot\rangle:= \langle\cdot\rangle_{x,\beta}$ denote expectation under the probability
measure proportional to $e^{\beta X^x_{n,p}(s)}\nu^{\otimes n}(ds)$. With $\varphi$ defined as~\eqref{eq:define_varphi},
\begin{align}\label{eq:derivatives}
\partial_jF_{n,\beta}=\frac1{\sqrt{nN_n}}\langle\varphi_j\rangle, \quad 
\partial_j^2F_{n,\beta} =\frac\beta{N_n}\bigl(\langle\varphi_j^2\rangle
                           -\langle\varphi_j\rangle^2\bigr), \quad 
\partial_j^3F_{n,\beta} =\frac{\beta^2\sqrt n}{N_n^{3/2}}
       \left\langle(\varphi_j-\langle\varphi_j\rangle)^3\right\rangle.
\end{align}
Since $|\varphi_j|\le 1$, we obtain, uniformly in $j$ and $x$,
\begin{equation}\label{eq:derivative-bounds}
|\partial_j^2F_{n,\beta}|\le\frac\beta{N_n}=:D_2,
\qquad
|\partial_j^3F_{n,\beta}|\le
\frac{8\beta^2\sqrt n}{N_n^{3/2}}=:D_3.
\end{equation}

We adapt~\cite[Theorem~1.1]{chatterjee2005simple} as follows. Choose $\alpha$ and $\alpha'$ independently and define
\begin{equation}
h_j(t)=F_{n,\beta}(\alpha'_1,\ldots,\alpha'_{j-1},t,
                         \alpha_{j+1},\ldots,\alpha_{m_n}).
\end{equation}
By a Taylor series expansion,
$h_j(t)=h_j(0)+t h_j'(0)+t^2h_j''(0)/2+R_j(t)$, with
\begin{equation}
|R_j(t)|\le\min\{D_2t^2,\ D_3|t|^3/6\}.
\end{equation}
using~\eqref{eq:derivative-bounds}. Fix $\varepsilon>0$ and $K=\varepsilon n^{(p-1)/2}$. Then we have
\begin{align}\label{eq:taylor-comparison}
\Delta_{n,\beta}:=\Big|\E F_{n,\beta}(\alpha')-\E F_{n,\beta}&(\alpha)\Big| \le D_2\sum_{j=1}^{m_n}\Bigg\{\E[(\alpha'_j)^2;|\alpha'_j|>K] +\E[\alpha_j^2;|\alpha_j|>K]\Bigg\} \nonumber\\
&+\frac{D_3}{6}\sum_{j=1}^{m_n}
 \Bigg\{\E[|\alpha'_j|^3;|\alpha'_j|\le K]
              +\E[|\alpha_j|^3;|\alpha_j|\le K]\Bigg\}
\end{align}
The second summand above can be bounded using
$|z|^3 1_{\{|z|\le K\}}\le Kz^2$ to obtain
\begin{equation}
\frac{D_3}{6}\sum_{j=1}^{m_n}
 \E[|\alpha'_j|^3;|\alpha'_j|\le K]
\le C\beta^2\varepsilon\frac{n^{p/2}}{\sqrt{N_n}}
\le C'\beta^2\varepsilon,
\end{equation}
since $N_n/n^p\to1/p!$. Thus, with $G\sim N(0,1)$, \eqref{eq:taylor-comparison} implies
\begin{equation}\label{eq:truncation-bound}
\Delta_{n,\beta}
\le \beta L_n(\varepsilon)
 +3^p\beta\,\E[G^2;|G|>K]
 +C_p\beta^2\left(\varepsilon+\sqrt{\frac n{N_n}}\right),
\end{equation}
where $L_n(\varepsilon)$ is defined by~\eqref{eq:lindeberg-assumption}, with $L_n(\varepsilon) \rightarrow 0$ by the hypothesis of the Proposition. For fixed $\beta>0$, taking $n\rightarrow \infty$ followed by $\varepsilon \rightarrow 0$ implies $\Delta_{n,\beta}\longrightarrow 0$.

By Lemma~\ref{lem:annealing}, we know $\kappa:= \sup_n \mathbb E M_n(\alpha) <\infty$. Using~\eqref{eq:cap-bound}, we have
\begin{align*}
    &\Big(1-\frac{\delta}{2}\Big)^p \mathbb E M_n(\alpha')-\frac1\beta\log\frac2\delta
\le \mathbb E F_{n,\beta}(\alpha')\le \mathbb E F_{n,\beta}(\alpha)+ \Delta_{n,\beta} \le \E M_n(\alpha)+\Delta_{n,\beta} \\
& \implies \mathbb E M_n(\alpha') -\mathbb E M_n(\alpha) \le \Bigg(\Big(1-\frac{\delta}{2}\Big)^{-p}-1\Bigg) \E M_n(\alpha) + \Big(1-\frac{\delta}{2}\Big)^{-p} \Big(\frac1\beta\log\frac2\delta + \Delta_{n,\beta}\Big)
\end{align*}
Similarly, 
\begin{align*}
    \mathbb E M_n(\alpha) -\mathbb E M_n(\alpha') \le \Bigg(1-\Big(1-\frac{\delta}{2}\Big)^{p}\Bigg)\E M_n(\alpha)+ \frac1\beta\log\frac2\delta + \Delta_{n,\beta}
\end{align*}
Combining the above two displays yield
\begin{align*}
    &|\E M_n(\alpha)- \E M_n(\alpha')|\\ 
&\le \kappa \left(\Big(1-\frac{\delta}{2}\Big)^{-p}-1\right) + \Big(1-\frac{\delta}{2}\Big)^{-p} \Big(\frac{1}{\beta}\log \frac{2}{\delta} + \Delta_{n,\beta}\Big)
\end{align*}
Choosing $\delta=\beta^{-1/2}$ and taking $n\rightarrow \infty$, followed by $\beta \rightarrow \infty$ proves that $\lim\limits_n|\E M_n(\alpha)- \E M_n(\alpha')| \rightarrow 0$.

Finally, to obtain concentration of $M_n(\alpha')$, we invoke Efron-Stein's inequality. Let $\alpha'^{(j)}$ be obtained from $\alpha'$ by replacing $\alpha'_j$ with an independent copy $\widetilde\alpha'_j$. By definition of $M_n(\alpha')$, we obtain
\begin{equation}
|M_n(\alpha')-M_n(\alpha'^{(j)})|
\le\frac{|\alpha'_j-\widetilde\alpha'_j|}{\sqrt{nN_n}}.
\end{equation}
Therefore, by the Efron--Stein inequality, we have
\begin{equation}\label{eq:concentration}
\Var\bigl(M_n(\alpha')\bigr)
\le\frac1{2nN_n}\sum_{j=1}^{m_n}
 \E(\alpha'_j-\widetilde\alpha'_j)^2
=\frac{3^p}{n}.
\end{equation}
Hence, $M_n(\alpha')-\E M_n(\alpha') \rightarrow 0$ in probability. Invoking Theorem~\ref{thm:main}, we complete the proof of the result.
\end{proof}

\section*{Acknowledgements}
I thank Rohan Sarkar, Samriddha Lahiry, and Subhabrata Sen for numerous helpful discussions. GPT was used to review the manuscript, check mathematical arguments, and assist with copyediting. All AI-generated content was carefully verified, and I take full responsibility for its correctness.

\bibliographystyle{alpha}
\bibliography{ref}

\end{document}